\documentclass{myaptpub}

\usepackage[left=2cm,right=2cm,top=2cm,bottom=2.5cm]{geometry}
\usepackage[english]{babel}
\usepackage{amsfonts,amsmath,amssymb,bbm,booktabs,comment,enumitem,graphicx,setspace,url}

\usepackage{natbib} 
\newcommand{\R}{\mathbb{R}}				% Real numbers
\newcommand{\BB}{\mathbb{B}}				% Ball
\newcommand{\Sphere}{\mathbb{S}}                        % Sphere
\newcommand{\EE}{\mathbb{E}}			% Expectations
\newcommand{\V}{\textrm{Var}}			% Variance
\newcommand{\C}{\textrm{Cov}}			% Covariance
\newcommand{\dd}{\mathrm{d}}			% differential d
\newcommand{\ang}{d}     	                % angle
\newcommand{\Levy}{\textrm{L\'{e}vy }}	        % Levy
\newcommand{\one}{\ensuremath{\mathbbm{1}}}     % indicator function
\newcommand{\abs}[1]{\left\vert#1\right\vert}

\def\ee{{\rm{e}}}

\def\dd{{\,\rm{d}}}

\newtheorem{algorithm}{Algorithm}

\authornames{L. V. Hansen {\em et al.}}
\shorttitle{Gaussian random particles}

\begin{document}

\title{Gaussian random particles with flexible Hausdorff dimension}

\begin{singlespace}

\authorone[Varde College]{Linda V.~Hansen} 
\addressone{Varde College, Frisvadvej 72, 6800 Varde, Denmark. E-mail: \url{LV@varde-gym.dk}}

\authortwo[Norwegian Computing Center]{Thordis L.~Thorarinsdottir}
\addresstwo{Norwegian Computing Center, P.O. Box 114 Blindern, 0314 Oslo, Norway. E-mail: \url{thordis@nr.no}}

\authorthree[Heidelberg Institute for Theoretical Studies]{Evgeni Ovcharov}
\addressthree{Heidelberg Institute for Theoretical Studies, Schloss-Wolfsbrunnenweg 35, 
  69118 Heidelberg, Germany. E-mail: \url{evgeni.ovcharov@h-its.org}}

\authorfour[Heidelberg Institute for Theoretical Studies and Karlsruhe Institute of Technology]{Tilmann Gneiting}
\addressfour{Heidelberg Institute for Theoretical Studies, Schloss-Wolfsbrunnenweg 35, 
  69118 Heidelberg, Germany. E-mail: \url{tilmann.gneiting@h-its.org}}

\authorfive[Penn State University]{Donald Richards}
\addressfive{Department of Statistics, Penn State University, 326
  Thomas Building, University Park, PA 16802, USA. E-mail: \url{richards@stat.psu.edu}}

\end{singlespace}

%%%%%%%%%%%%%%%%%%%%%%%%%%%%%%%%%
%      Abstract                 %
%%%%%%%%%%%%%%%%%%%%%%%%%%%%%%%%%

\begin{abstract} 

Gaussian particles provide a flexible framework for modelling and
simulating three-dimensional star-shaped random sets.  In our
framework, the radial function of the particle arises from a kernel
smoothing, and is associated with an isotropic random field on the
sphere.  If the kernel is a von Mises--Fisher density, or uniform on a
spherical cap, the correlation function of the associated random field
admits a closed form expression.  The Haus\-dorff dimension of the
surface of the Gaussian particle reflects the decay of the correlation
function at the origin, as quantified by the fractal index.  Under
power kernels we obtain particles with boundaries of any Hausdorff
dimension between 2 and 3.

\end{abstract}

%%%%%%%%%%%%%%%%%%%%%%%%%%%%%%%%%
%      Key words                %
%%%%%%%%%%%%%%%%%%%%%%%%%%%%%%%%%

\vspace{3mm}

\vspace{3mm}

%%%%%%%%%%%%%%%%%%%%%%%%%%%%%%%%%
%      Introduction             %
%%%%%%%%%%%%%%%%%%%%%%%%%%%%%%%%%

\section{Introduction}  \label{sec:introduction} 

Mathematical models for three-dimensional particles have received
great interest in astronomy, botany, geology, material science, and
zoology, among many other disciplines.  While some particles such as
crystals have a rigid shape, many real-world objects are star-shaped,
highly structured, and stochastically varying \citep{Wicksell:1925,
StoyanStoyan1992}.  As a result, flexible yet parsimonious models for
star-shaped random sets have been in high demand.
\citet{GrenanderMiller1994} proposed a model for two-dimensional
featureless objects with no obvious landmarks, which are represented
by a deformed polygon along with a Gaussian shape model.  This was
investigated further by \citet{KentDrydenAnderson2000} and
\citet{HobolthKentDryden2002}, and a non-Gaussian extension was
suggested by \citet{HobolthPedersenJensen2003}.  \citet{Miller1994}
proposed an isotropic deformation model that relies on spherical
harmonics and was studied by \citet{Hobolth2003}, where it was applied
to monitor tumour growth.  A related Gaussian random shape model was
studied by \citet{Muinonen&1996} and used by \citet{Munoz&2007} to
represent Saharan desert dust particles.

In this paper we propose a flexible framework for modelling
three-dimensional star-shaped particles, where the radial function is
a random field on the sphere that arises through a kernel smoothing.
Specifically, let $Y\subset \R^3$ be a
three-dimensional compact set, which is star-shaped with respect to an
interior point $o$.  Then there is a one-to-one correspondence between
the set $Y$ and its radial function $X = \{X(u) \, : \, u \in
\Sphere^2 \}$, where
\[
X(u) = \max \{ r \geq 0 : o + ru \in Y \}, \qquad u \in \Sphere^2, 
\]
with $\Sphere^2 = \{ x \in \R^3 : \| x \| = 1 \}$ denoting the unit
sphere in $\R^3$.  We model $X$ as a real-valued random field on
$\Sphere^2$ via a kernel smoothing of a Gaussian measure, in that
\begin{equation}  \label{eq:model} 
X(u) = \int_{\Sphere^2} K(v,u) \, L(\dd v), \qquad u \in \Sphere^2, 
\end{equation}
where $K : \Sphere^2 \times \Sphere^2 \to \bar{\R}$ is a suitable
kernel function, and $L$ is a Gaussian measure on the Borel subsets of
$\Sphere^2$.  That is, 
$L(A) \sim \mathcal{N} \! \left( \, \mu \, \lambda(A), \sigma^2 \lambda(A) \right)$
with parameters $\mu \in \R$ and $\sigma^2 > 0$, 
where $\lambda(A)$ denotes the surface measure of a Borel set $A
\subseteq \Sphere^2$, with $\lambda(\Sphere^2) = 4\pi$.

If $X$ were a nonnegative process, the random particle could be described as the set
\[
Y = \bigcup_{u \in \Sphere^2} \left\{ o + ru : 0 \leq r \leq X(u) \right\}	
\subset \R^3, 
\]
so that the particle contains the centre $o$, which without loss of
generality can be assumed to be the origin, and the distance in
direction $u$ from $o$ to the particle boundary is given by $X(u)$.  A
potentially modified particle $Y_c$ arises in the case of a general,
not necessarily nonnegative process, where we replace $X(u)$ by
$X_c(u) = \max(c, X(u))$ for some $c > 0$.  We call $Y$ or $Y_c$ a
Gaussian particle, with realisations being illustrated in Figure
\ref{fig:Gaussian}.  The Gaussian particle framework is a special 
case of the linear spatio-temporal \Levy model proposed by
\citet{Jonsdottir2008} in the context of tumour growth.
Alternatively, it can be seen as a generalisation and a
three-dimensional extension of the model proposed by
\citet{HobolthPedersenJensen2003}, while also being a generalisation
of the Gaussian random shape models of \citet{Miller1994} and
\citet{Muinonen&1996}.

\begin{figure}[t]  
\centering
\includegraphics[height=4cm]{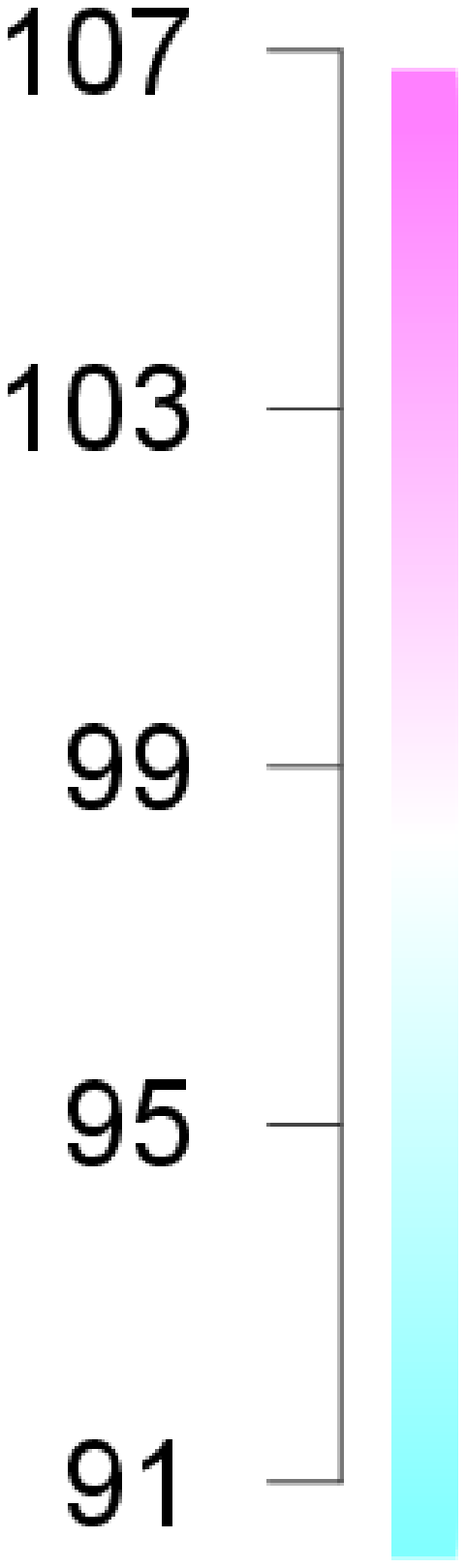}
\hspace{0.7cm}
\includegraphics[height=4cm]{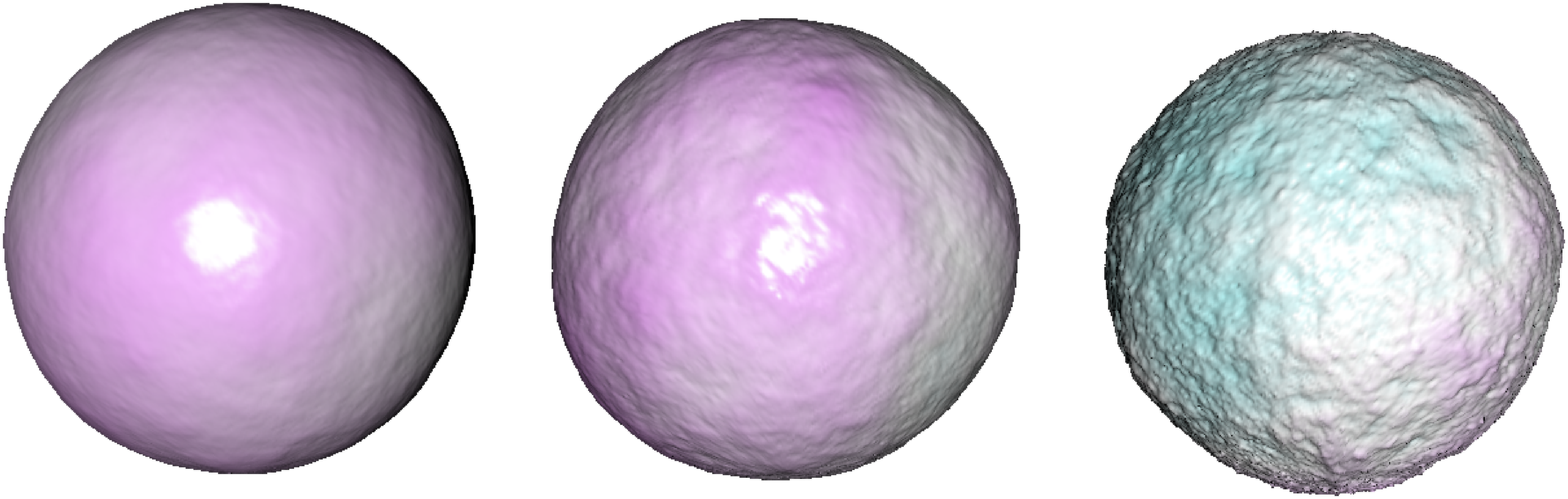}
\caption{Gaussian particles with mean $\mu_X = 100$ and variance
  $\sigma^2_X = 10$, using the power kernel
  \eqref{eq:power} with $q = 0.05$ (left), $q = 0.25$ (middle) and $q
  = 0.5$ (right).  The Hausdorff dimension of the particle
  surface equals $2 + q$.
\label{fig:Gaussian}}
\bigskip
\end{figure}

The realisations in Figure \ref{fig:Gaussian} demonstrate that the
boundary or surface of a Gaussian particle allows for regular as well
as irregular behaviour.  The roughness or smoothness of the surface in
the limit as the observational scale becomes infinitesimally fine can
be quantified by the Hausdorff dimension, which for a surface in
$\R^3$ varies between 2 and 3, with the lower limit corresponding to a
smooth, differentiable surface, and the upper limit corresponding to
an excessively rough, space-filling surface \citep{Falconer1990}.  The
Hausdorff dimension of the surface of an isotropic Gaussian particle
is determined solely by the behaviour of the correlation function of
the associated random field on the sphere.  We investigate the
properties of Gaussian particles under parametric families of
isotropic kernel functions, including power kernels, and kernels that
are proportional to von Mises--Fisher densities
\citep{Fisher_etal_1987}, or uniform on spherical caps.  Under power
kernels we obtain particles with boundaries of any Hausdorff dimension
between 2 and 3.  Von Mises--Fisher and uniform kernels generate
Gaussian particles with boundaries of Hausdorff dimension 2 and 2.5,
respectively.

The remainder of the paper is organised as follows.  Section
\ref{sec:Levy} recalls basic properties of the radial function in the
Gaussian particle model \eqref{eq:model}.  In
Section~\ref{sec:Hausdorff} we show how to derive the Hausdorff
dimension of an isotropic Gaussian particle from the infinitesimal
behaviour of the correlation function of the underlying random field
at the origin.  Section \ref{sec:kernels} introduces the
aforementioned families of isotropic kernels and discusses the
properties of the associated correlation functions and Gaussian
particles, with some technical arguments referred to an appendix.
Section \ref{sec:simulation} presents a simulation algorithm and
simulation examples, including a case study on celestial bodies and a
discussion of planar particles.  The paper ends with a discussion in
Section \ref{sec:discussion}.

%%%%%%%%%%%%%%%%%%%%%%%%%%%%%%%%%
%      Preliminaries            %
%%%%%%%%%%%%%%%%%%%%%%%%%%%%%%%%%

\section{Preliminaries}  \label{sec:Levy}

The properties of the random function \eqref{eq:model} that
characterises a Gaussian particle process depend on the kernel
function $K$.  We assume that $K$ is isotropic, in that $K(v,u) =
k(\ang(v,u))$ depends on the points $v, u \in \Sphere ^2$ through
their great circle distance $\ang(v,u) \in [0,\pi]$ only.  As 
$\ang(v,u) = \arccos(u \cdot v)$, this is equivalent to assuming that 
the kernel depends on the inner product $u \cdot v$ only.  Results of
\citet{Jonsdottir2008} in concert with the rotation invariance property 
following from an isotropic kernel imply that the mean function 
$\EE(X(u))$ and the variance function $\V(X(u))$ are constant, that is,
\[
\mu_X =\EE(X(u)) = \mu \, c_1 
\qquad \textrm{and} \qquad
\sigma^2_X = \V(X(u)) = \sigma^2 \, c_2
\]
for $u \in \Sphere^2$, where we assume that 
\[
c_n = \int_{\Sphere^2} k(\ang(v,u))^n \, \dd v
\]
is finite for $n = 1, 2$.  

Note that $X$ is a stochastic process on the sphere \citep{Jones1963}, whose 
covariance function is given by
\[
\C(X(u_1),X(u_2)) = \sigma^2 \int_{\Sphere^2} k(\ang(v,u_1)) \, k(\ang(v,u_2)) \, \dd v, 
\qquad u_1, u_2 \in \Sphere^2, 
\]
Under an isotropic kernel, the random field $X$ is isotropic as well, and it
is readily seen that $\textrm{Corr}(X(u_1),X(u_2)) =
C(\ang(u_1,u_2))$, where
\begin{equation}  \label{eq:C}
C(\theta) = \frac{2}{c_2} \int_0^\pi \int_0^\pi k(\eta) \, 
k(\arccos(\sin\theta \sin\eta \cos\phi + \cos\theta \cos\eta)) 
\, \dd \phi \, \sin\eta \, \dd\eta, \qquad 0 \leq \theta \leq \pi, 
\end{equation}
is the correlation function of the random field $X$.  As recently
shown by \cite{Ziegel2014}, any continuous isotropic correlation
function on a sphere admits a representation of this form.

%%%%%%%%%%%%%%%%%%%%%%%%%%%%%%%%%
%      Hausdorff dimension      %
%%%%%%%%%%%%%%%%%%%%%%%%%%%%%%%%%

\section{Hausdorff dimension}  \label{sec:Hausdorff}

The Hausdorff dimension of a set $Z \subset \R^d$ is defined as
follows \citep{Hausdorff1918}.  For $\epsilon > 0$, an
$\epsilon$-cover of $Z$ is a countable collection $\{B_i : i =
1,2,\ldots\}$ of balls $B_i \subset \R^d$ of diameter $|B_i|$ less
than or equal to $\epsilon$ that covers $Z$.  With
\[
H^{\delta}(Z) = \lim_{\epsilon \rightarrow 0} \; 
\inf \left\{ \: {\textstyle \sum} |B_i|^\delta : \{B_i : i = 1,2,\ldots\} 
\textup{ is an $\epsilon$-cover of $Z$} \right\}
\]
denoting the $\delta$-dimensional Hausdorff measure of $Z$, there
exists a unique nonnegative number $\delta_0$ such that $H^{\delta}(Z)
= \infty$ if $\delta < \delta_0$ and $H^{\delta}(Z) = 0$ if $\delta >
\delta_0$.  This number $\delta_0$ is the Hausdorff dimension of the
set $Z$.  Note that we have defined the Hausdorff measure using
coverings with balls.  This approach is consistent with the treatments
given by \citet{Adler1981} and \citet{HallRoy1994} and simplifies the
presentation.  

As $X$ is a kernel smoothing of a Gaussian measure, $X$ has Gaussian
finite dimensional distributions and thus is a Gaussian process.
While there is a wealth of results on the Hausdorff dimension of the
graphs of stationary Gaussian random fields on Euclidean spaces, which
is determined by the infinitesimal behaviour of the correlation
function at the origin, as formalised by the fractal index
\citep{HallRoy1994, Adler1981}, we are unaware of any extant results
for the graphs of random fields on spheres, or for the surfaces of
star-shaped random particles.

We now state and prove such a result.  Toward this end, we say that an
isotropic random field $X$ on the sphere with correlation function $C
: [0,\pi] \to \R$ has fractal index $\alpha > 0$ if there exists a constant
$b > 0$ such that
\begin{equation}  \label{eq:fractal.index} 
\lim_{\theta \downarrow 0} \frac{C(0) - C(\theta)}{\theta^\alpha} =  b.
\end{equation}
The fractal index exists for essentially all correlation functions of
practical interest, and it is always true that $\alpha \in (0,2]$. 
To see this, suppose that $C : [0,\pi] \to \R$ is an isotropic
correlation function on the two-dimensional sphere.  Clearly, $C$ also
is an isotropic correlation function on the circle, and its even,
$2\pi$ periodic continuation to $\R$ is a stationary correlation
function on the real line.  Therefore, the corresponding restriction
on Euclidean spaces \citep[p.~200]{Adler1981} applies, in that $\alpha
\in (0,2]$.

The following theorem relates the Hausdorff dimension of the graph of
an isotropic Gaussian random field $X$ on the sphere $\Sphere^2$ to
its fractal index.  The proof employs stereographic projections that
allow us to draw on classical results in the Euclidean case.

\begin{theorem} \label{thm:Hausdorff}
Let\/ $X$ be an isotropic Gaussian random field on\/ $\Sphere^2$ with
fractal index\/ $\alpha \in (0,2]$.  Consider the random surface
\[ 
Z_c = \left\{ (u, X_c(u)) : u \in \Sphere^2 \right\} \! ,  
\]
where\/ $X_c(u) = \max(c,X(u))$ with\/ $c > 0$.  Then with probability
one either of the following alternatives holds:
\begin{enumerate}
\item[(a)] If\/ $\max_{u \in \Sphere^2} X(u) \leq c$, the realisation
  of\/ $Z_c$ is the sphere with radius\/ $c$ and so its Hausdorff
  dimension is 2.
\item[(b)] If\/ $\max_{u \in \Sphere^2} X(u) > c$, the realisation
  of\/ $Z_c$ has Hausdorff dimension\/ $3 - \frac{\alpha}{2}$.
\end{enumerate}
\end{theorem}

\begin{proof}
The claim in alternative (a) is trivial.  To prove the statement in
alternative (b), we assume without loss of generality that $X(u_0) >
c$, where $u_0 = (0,0,1)$.  The sample paths of $X$ are continuous
almost surely according to \citet[Theorem 7.2]{Gangolli}.  Thus, there
exists an $\epsilon \in (0,\frac{1}{2})$ such that $X(u) > c$ for $u$
in the spherical cap $\Sphere^2_\epsilon = \{ u \in \Sphere^2 :
\ang(u,u_0) \leq \epsilon \}$ of radius $\epsilon$ centred at $u_0$.
Let $\Pi : \Sphere^2_\epsilon \to \BB_\epsilon$ denote a stereographic
projection that maps $(0,0,1)$ to $(0,0)$, where $\BB_\epsilon = \{ x =
(x_1,x_2) \in \R^2 : x_1^2 + x_2^2 \leq \epsilon^2 \}$.  A
stereographic projection is a local diffeomorphism, $\Pi$ thus is
differentiable and has a differentiable inverse $\Pi^{-1}$, which is
locally bi-Lipschitz \citep{docarmo}.  We may therefore assume that
$\epsilon$ is small enough so that for all $x, x' \in \BB_\epsilon$
there exists a constant $A \geq 1$ with
\begin{equation}  \label{eq:bi.Lipschitz}
\frac{1}{A} \| x - x' \| \leq \| \Pi^{-1}(x) - \Pi^{-1}(x') \| \leq A \| x - x' \|, 
\end{equation} 
where $\| \cdot \|$ denotes the Euclidean norm on $\R ^2$ or $\R^3$,
respectively.  Without loss of generality, we may in the following 
consider conditional probabilities which depend on the choice of $u_0$ and 
$\epsilon$.  Let the Gaussian random field $W$ on $\BB_\epsilon
\subset \R^2$ be given by $W(x) = X(\Pi^{-1}(x))$.  From
\citet[Theorem 5.1]{XueXiao2011}, see also Chapter 8 in
\citet{Adler1981}, the graph $\textrm{Gr} \, W = \{ (x, W(x)) : x \in
\BB_\epsilon \}$ has Hausdorff dimension $3 - \frac{\alpha}{2}$ almost
surely if there exists a constant $M_0 > 1$ such that
\begin{equation}  \label{eq:XueXiao}
\frac{1}{M_0} \sum_{j=1}^2 |x_j -x'_j|^\alpha 
\leq \EE (W(x)-W(x'))^2 
\leq M_0 \sum_{j=1}^2 |x_j -x'_j|^\alpha
\end{equation}
for all $x, x' \in \BB_\epsilon$.  Letting $\vartheta(x,x') =
\ang(\Pi^{-1}(x), \Pi^{-1}(x'))$, we have 
\begin{equation}  \label{eq:W}
\EE (W(x)-W(x'))^2 = 2 \sigma_X^2 \left[ C(0) - C(\vartheta(x,x')) \right] \! , 
\end{equation}
where $C : [0,\pi] \to \R$ is the correlation function of the
isotropic random field $X$.  As chord length and great circle distance
are bi-Lipschitz equivalent metrics, there exists a constant $B > 1$
such that
\begin{equation}  \label{eq:chord.arc}
\frac{1}{B} \, \|\Pi^{-1}(x) - \Pi^{-1}(x')\| 
\leq \vartheta(x,x')
\leq B \, \|\Pi^{-1}(x) - \Pi^{-1}(x')\|.
\end{equation}
As the random field $X$ is of fractal index $\alpha$, there exists a
constant $M_1 > 0$ such that
\[
\sum_{j=1}^2 | x_j - x'_j |^{\alpha} 
\leq 2^{1 - \tfrac{\alpha}{2}} \|x-x'\|^\alpha 
\leq 2^{1 - \tfrac{\alpha}{2}} A^\alpha B^\alpha \, \vartheta(x,x')^\alpha 
\leq M_1 \left[ C(0) - C(\vartheta(x,x')) \right]
\]
for $x, x' \in \BB_\epsilon$ and $\epsilon > 0$ sufficiently small,
where the first estimate is justified by Jensen's inequality and the
second by \eqref{eq:bi.Lipschitz} and \eqref{eq:chord.arc}.
Similarly, there exists a constant $M_2 > 0$ such that
\[
M_2 \left[ C(0) - C(\vartheta(x,x')) \right] \leq \sum_{j=1}^2 | x_j - x'_j |^{\alpha}
\]
for all $x, x' \in \BB_\epsilon$ and $\epsilon > 0$ sufficiently small.
In view of equation \eqref{eq:W}, this proves the existence of a
constant $M_0 > 1$ such that \eqref{eq:XueXiao} holds, given that
$\epsilon > 0$ is sufficiently small.

Now, consider the mapping $\zeta$ from $\BB_\epsilon \times \R$ to
$\Sphere^2_\epsilon \times \R$ defined by $\zeta(x,r) = (\Pi^{-1}(x),
r)$, so that $\zeta(\textrm{Gr} \, W) = \{ (u, X(u)) : u \in
\Sphere^2_\epsilon \}$.  The identity
\[
\| \zeta(x,r) - \zeta (x',r' )\|^2 = \| \Pi^{-1}(x) - \Pi^{-1}(x') \|^2 + |r - r'|^2.
\]
along with \eqref{eq:bi.Lipschitz} implies $\zeta$ to be bi-Lipschitz.
Therefore by Proposition 3.3 of \citet{Falconer1990}, the partial
surface $\{ (u, X(u)) : u \in \Sphere^2_\epsilon \}$ has Hausdorff
dimension $3 - \frac{\alpha}{2}$ almost surely.  Invoking the
countable stability property \citep[p.~49]{Falconer1990}, we see that the
full surface $Z_c = \{ (u, X_c(u)) : u \in \Sphere^2_\epsilon \}$ also
has Hausdorff dimension $3 - \frac{\alpha}{2}$ almost surely.
\end{proof}

%%%%%%%%%%%%%%%%%%%%%%%%%%%%%%%%%
%      Isotropic kernels        %
%%%%%%%%%%%%%%%%%%%%%%%%%%%%%%%%%

\section{Isotropic kernels}  \label{sec:kernels}

It is often desirable that the surface of the particle process
possesses the same Hausdorff dimension as that of the real-world
particles to be emulated \citep{Mandelbrot1983, OrfordWhalley1983,
Turcotte1987}.  With this in mind, we introduce and study three
one-parameter families of isotropic kernels for the Gaussian particle
process \eqref{eq:model}.  The families yield interesting correlation
structures, and we study the asymptotic behaviour at zero, which
determines the Hausdorff dimension of the Gaussian particle surface.

%%%%%%%%%%%%%%%%%%%%%%%%%%%%%%%%%%%%%%
%      The von Mises-Fisher kernel   %
%%%%%%%%%%%%%%%%%%%%%%%%%%%%%%%%%%%%%%

\subsection{Von Mises--Fisher kernel}

Here, we consider $k$ to be the unnormalised von Mises--Fisher
density, 
\[
k(\theta) = \ee^{a \cos \theta}, \qquad 0 \leq \theta \leq \pi, 
\]
with parameter $a > 0$.  The von Mises--Fisher density with parameter
$a > 0$ is widely used in the analysis of spherical data
\citep{Fisher_etal_1987}, and in this context $a$ is called the
precision.  Straightforward calculations show that 
\[
C(\theta) = \frac{2}{\sinh(2a)} \, 
\frac{\sinh \! \left( a \sqrt{2(1+\cos\theta)} \right)}{\sqrt{2(1+\cos\theta)}},
\qquad 0 \leq \theta \leq \pi, 
\]
from which it is readily seen that the fractal index is $\alpha = 2$.
The surfaces of the corresponding Gaussian particles are smooth
and have Hausdorff dimension 2, independently of the value of the
parameter $a \in \R$. 

%%%%%%%%%%%%%%%%%%%%%%%%%%%%%%%%%
%      Uniform kernel           %
%%%%%%%%%%%%%%%%%%%%%%%%%%%%%%%%%

\subsection{Uniform kernel}

We now let the kernel $k$ be uniform, in that
\[
k(\theta) = \one(\theta \leq r), \qquad 0 \leq \theta \leq \pi, 
\]
with cut-off parameter $r \in (0,\frac{\pi}{2}]$.  As shown in the
appendix of \cite{Tovchigrechko}, the associated correlation function is
\begin{eqnarray*} 
\lefteqn{\hspace{-5mm} C(\theta) = \frac{1}{\pi \, (1 - \cos r)} 
\left( \pi - \arccos \left( \frac{\cos \theta -\cos^2 r}{1 - \cos^2 r} \right) \right.} \\ 
&& \hspace{10mm} \left. \vphantom{\csc^2}
- \, 2 \cos r \arccos \left( \cot r \, \frac{1- \cos \theta}{\sin \theta}  \right) \right)
\one(\theta \leq 2r), \qquad 0 \leq \theta \leq \pi. 
\end{eqnarray*} 
In particular, if $r = \frac{\pi}{2}$ then $C(\theta) = 1 -
\frac{\theta}{\pi}$ decays linearly throughout.  Taylor expansions
imply that the correlation function has fractal index $\alpha = 1$ for
all $r \in (0,\frac{\pi}{2})$, so that the corresponding Gaussian
particles have non-smooth boundaries of Hausdorff dimension
$\frac{5}{2}$.

%%%%%%%%%%%%%%%%%%%%%%%%%%%%%%%%%
%      Power kernel   %
%%%%%%%%%%%%%%%%%%%%%%%%%%%%%%%%%

\subsection{Power kernel}\label{sec: power kernel}

Our third example is the power kernel where the isotropic kernel 
$k$ is defined as 
\begin{equation}  \label{eq:power}
k(\theta) = \left( \frac{\theta}{\pi} \right)^{-q} - 1, \qquad 0 < \theta \leq \pi, 
\end{equation}
with power parameter $q \in (0,1)$.  The associated
correlation function \eqref{eq:C} takes the form
\begin{equation}  \label{eq:power.correlation}
C(\theta) = \frac{2}{c_2} \int_0^\pi \! \left( \pi^q \lambda^{-q} - 1 \right) \sin\lambda 
\int_{A(\lambda)} \! \left( \pi^q a(\theta,\lambda,\phi)^{-q} -1 \right) \dd\phi \, \dd\lambda,
\end{equation}
where 
\[
t(\theta,\lambda,\phi) = \sin \theta \sin \lambda \cos \phi + \cos\theta \cos\lambda,
\qquad 
a(\theta,\lambda,\phi) = \arccos t(\theta,\lambda,\phi), 
\]
and 
\[
A(\lambda) = \{ \phi \in [0,\pi] : 0 < a(\theta,\lambda,\phi) \leq \pi \}.  
\]
The normalising constant $c_2$ is here given by 
\[
c_2 = 
2\pi \int_0^\pi (\pi^q \lambda^{-q} -1)^2 \sin \lambda {\;\rm d}\lambda 
= \frac{(q)_3}{6} \sum_{j=0}^\infty \frac{(-1)^j \pi^{2j+3}}{(2j+1)!} \frac{1}{(1-q+j)_3}, 
\]
where $(a)_3 \equiv a(a+1)(a+2)$.  This expression for $c_2$ is
obtained by expanding $\sin \lambda$ in a Maclaurin series and then
integrating the series termwise.

Our next result shows that the correlation function
\eqref{eq:power.correlation} has fractal index $\alpha = 2 - 2q$, so
that the corresponding Gaussian particles have surfaces with Hausdorff
dimension $2 + q$, as illustrated in Figure \ref{fig:Gaussian}.

\begin{theorem} \label{thm:power}
If\/ $0 < q < 1$, the correlation function \eqref{eq:power.correlation} satisfies 
\begin{equation}  \label{eq:as} 
\lim_{\theta \downarrow 0} \frac{C(0) - C(\theta)}{\theta^{2-2q}} =  b_q, 
\end{equation}
where 
\begin{align}  
b_q & = \frac{2 \pi^{2q}}{c_2} 
\int_0^{\infty} x^{1-q} \int_0^\pi \left( x^{-q} - 
\left( x^2 + 1 - 2x\cos\phi \right)^{-q/2} \right) \dd\phi \, \dd x \label{eq:bq} \\
& = \frac{\pi^{2q+1}}{c_2 (1-q)^2} \,  
\frac{\Gamma(1 - \tfrac{1}{2} q)^2 \, \Gamma(q)}{\Gamma(\tfrac{1}{2} q)^2 \, \Gamma(1-q)}.
\rule{0mm}{8mm} \label{eq:explicit}
\end{align}
In particular, the correlation function has fractal index\/ $\alpha = 2 - 2q$. 
\end{theorem}

We defer the proof of this result to the Appendix.  The power kernel
\eqref{eq:power} has a negative exponent and thus is unbounded, which
may lead to unbounded particle realisations.  While values of $q < 0$
are feasible, they are of less interest, as the associated correlation
functions have fractal index $\alpha = 2$, thereby generating smooth
particles only.

%%%%%%%%%%%%%%%%%%%%%%%%%%%%%%%%%
%  Simulation examples          %
%%%%%%%%%%%%%%%%%%%%%%%%%%%%%%%%%

\section{Examples}  \label{sec:simulation}

Here, we demonstrate the flexibility of the Gaussian particle
framework in simulation examples.  First, we introduce a simulation
algorithm.  Then we simulate celestial bodies whose surface properties
resemble those of the Earth, the Moon, Mars, and Venus, as reported in
the planetary physics literature.  Furthermore, we study and simulate
the planar particles that arise from the two-dimensional version of
the particle model.

%%%%%%%%%%%%%%%%%%%%%%%%%%%%%%%%%
%      The simulation algorithm %
%%%%%%%%%%%%%%%%%%%%%%%%%%%%%%%%%

\subsection{Simulation algorithm}  \label{sec:algorithm}

To sample from the Gaussian particle model \eqref{eq:model}, we utilise
the property that the underlying measure is independently scattered.
Specifically, for every sequence $(A_n)$ of disjoint Borel subsets of
$\Sphere^2$, the random variables $L(A_n)$, $n = 1, 2, \dots$ are
independent and $L(\cup A_n) = \sum L(A_n)$ almost surely.
Let $(A_n)_{n=1}^N$ denote an equal area partition of $\Sphere^2$, so
that $\lambda(A_n) = 4\pi/N$ for $n = 1, \dots , N$.  The random
field $X$ in \eqref{eq:model} can then be decomposed into a sum of
integrals over the disjoint sets $A_n$, in that
\[
X(u) = \sum_{n=1}^N \int_{A_n} k(v,u) \, L(\dd v), \qquad u \in \Sphere^2.
\]
For $n=1, \dots , N$ fix any point $v_n \in A_n$.  We can then
approximate the random field $X$ by setting 
\[
x(u) = \sum_{n=1}^N k(v_n,u) \, L( A_n ), \qquad u \in \Sphere^2.
\]
Let us denote the multivariate normal joint distribution of $L(A_1),
\ldots, L(A_N)$ by $F_N$.  To simulate a realisation $y$ of the
particle $Y_c$, we use the following algorithm.

% \newpage
\begin{algorithm} \mbox{}
\begin{enumerate}[noitemsep]
\item Set\/ $M = M_1 M_2$, where\/ $M_1$ and\/ $M_2$ are positive
  integers, and construct a grid\/ $u_1, \ldots, u_M$ on\/
  $\Sphere^2$.  Using spherical coordinates, let\/ $u_m = (\theta_m,
  \phi_m)$ and put\/ $\theta_m = i\pi/M_1$ and\/ $\phi_m = 2\pi
  j/M_2$, where\/ $m = i M_2 + j$ for\/ $i = 0, 1, \dots$, $M_1 - 1$
  and\/ $j = 1, \dots, M_2$.
\item Apply the method of \citet{Leopardi2006} to construct an equal
  area partition\/ $A_1, \ldots, A_N$ of\/ $\Sphere^2$.
\item For\/ $n = 1, \ldots, N$, let\/ $v_n$ have spherical coordinates
  equal to the mid range of the latitudes and longitudes within\/
  $A_n$, respectively.
\item For\/ $n = 1, \ldots, N$, generate independent random variables\/ $L_n$
  from\/ $F_N$.
\item For\/ $m = 1, \ldots, M$, set\/ $x(u_m) = \max(c, \sum_{n=1}^N k(v_n,u_m) \, L_n)$.
\item Set\/ $y$ to be the triangulation of\/ $\{ (u_m,x(u_m)) : m = 1, \ldots, M \}$. 
\end{enumerate}
\end{algorithm}

The equal area partitioning algorithm of \citet{Leopardi2006} is a
recursive zonal partitioning algorithm.  That is, after appropriate
polar cap areas have been removed, the sphere is divided into
longitudinal zones, each of which is subsequently divided by latitude.
By construction, the equal areal partition cells are continuity sets
with respect to the intensity of the Gaussian measure, and in all our examples 
the kernel $k$ is continuous almost everywhere. 

This simulation procedure has been implemented in {\sc R} \citep{R},
and code is available from the authors upon request.  It can be
considered an analogue of the moving average method \citep{Oliver95,
CP02, HansenThorarinsdottir2013} for simulating Gaussian random fields
on Euclidean spaces.  In principle, $M$ and $N$ can take any positive
integer values.  However, the usual trade-off applies, in that the
quality of the realisations increases with $M_1$, $M_2$, and $N$, at
the expense of prolonged run times.  For the realisations in
Figures~\ref{fig:Gaussian}--\ref{fig:gamma}, we used $M_1 = 200$ and
$M_2 = 400$, or $M = 8 \times 10^4$, and $N = 10^5$.

%%%%%%%%%%%%%%%%%%%%%%%%%%%%%%%%%
%      Celestial bodies         %
%%%%%%%%%%%%%%%%%%%%%%%%%%%%%%%%%

\begin{table}[pt]  
\caption{Mean radius $r_0$, difference $d_+$ between maximal and mean
  radius, and difference $d_-$ between minimal and mean radius, for
  Venus, Dry Earth, the Moon, and Mars, in kilometres.  \label{tab:bodies}}
\begin{center}
\begin{tabular}{lrrrr}
\toprule
Body & Venus & Dry Earth & Moon & Mars \\
\midrule
$r_0$ & 6051.8 & 6367.2 & 1737.1 & 3389.5 \\
$d_+$ & 11.0 & 8.8 & 5.5 & 21.2 \\
$d_-$ & $-3.0$ & $-11.0$ & $-12.0$ & $-8.2$ \\
\bottomrule
\end{tabular}
\end{center}
\end{table}

\begin{figure}[pt]  
\begin{center}
\includegraphics[trim = 0 0 0 0, width=0.65\textwidth]{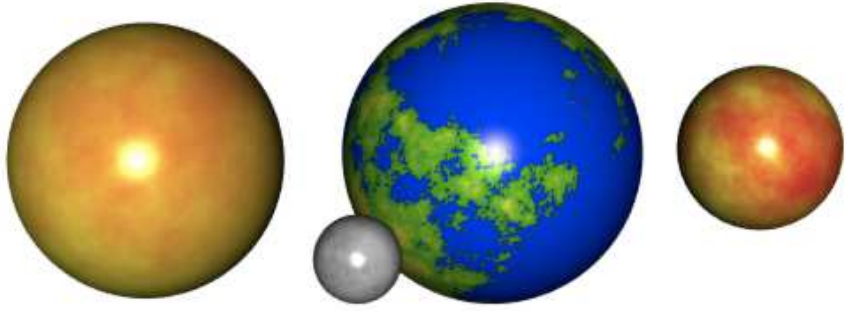}
\caption{Simulations of Venus, the Earth, the Moon, and Mars in true relative size.  
  \label{fig:bodies}}
\end{center}
\end{figure}

\begin{figure}[pt]  
\begin{center}
\includegraphics[trim = 0 0 0 0, width=0.7\textwidth]{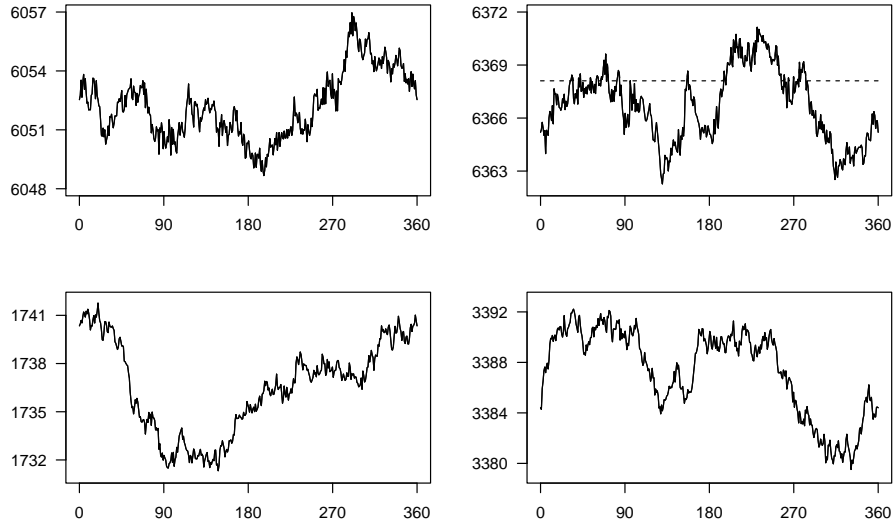}
\caption{Radial function along the equator for the simulated bodies in
  Figure \ref{fig:bodies} in kilometres.  Clockwise from upper left:
  Venus, the Earth with ocean level indicated by a dashed horizontal line, Mars, and the Moon.  
  \label{fig:equatorial}}
\end{center}
\bigskip
\end{figure}

\subsection{Celestial bodies}

The geophysical literature has sought to characterise the surface
roughness of the Earth and other celestial bodies in the solar system
via the Hausdorff dimension of their topography \citep{Mandelbrot1983,
  Kucinskas&1992}, with \citet{Turcotte1987} arguing that the
dimension is universal and equals about 2.5.  Here, we provide
simulated version of the planets Earth, Venus, and Mars, and of the
Moon, under the Gaussian particle model \eqref{eq:model}, with $k$
being the power kernel \eqref{eq:power}.  We set $q = \tfrac{1}{2}$,
which gives the desired fractal dimension for a Gaussian particle
surface, and choose the parameters $\mu = r_0/c_1$ and $\sigma^2 =
(d_+ - d_-)^2/c_2$ of the Gaussian measure such that they correspond
to reality.  For this we use the information listed in Table
\ref{tab:bodies}, which was obtained from \citet{Price1988},
\citet{JonesStofan2008}, and online sources.  The values concerning
the Earth describe `Dry Earth'; to simulate `Wet Earth' we make a
cut-off that corresponds to the Gaussian particle $Y_c$ with
truncation parameter $c = 6371$ kilometres.  In principle, we also
need to make a cut-off at $c=0$ for `Dry Earth', but this is
unnecessary in essentially all realizations.  In our simulation
algorithm, we use $M_1 = 200$, $M_2 = 400$, and $N = 10^6$ to obtain
the celestial bodies in Figure \ref{fig:bodies}.  The corresponding
radial functions along the equator are shown in Figure
\ref{fig:equatorial}.

%%%%%%%%%%%%%%%%%%%%%%%%%%%%%%%%%
%      2D study                 %
%%%%%%%%%%%%%%%%%%%%%%%%%%%%%%%%%

\subsection{Planar particles}  \label{sec:planar}

\begin{table}[t]  
\caption{Analytic form, parameter range, constants and associated
  fractal index for parametric families of isotropic kernels $k :
  [0,2\pi) \to \bar{\R}$ on the circle
    $\Sphere^1$.  \label{tab:planar}}
\begin{center}
\begin{tabular}{lccc}
\toprule
Kernel & von Mises--Fisher & Uniform &  Power \\
\midrule
Analytic Form & $k(\theta) = \ee^{a \cos\theta}$ & $k(\theta) = \one(\theta \leq r)$ & 
$k(\theta) = \left( \frac{\theta}{\pi} \right)^{-q} - 1 $ \rule{0mm}{8mm} \\
Parameter     & $a > 0$ & $r \in (0,\frac{\pi}{2}]$ & $q \in (-\frac{1}{2},0) \cup (0,\frac{1}{2})$ \rule{0mm}{8mm} \\
$c_1$         & $2\pi I_0(a)$ & $2r$ & $\displaystyle 2\pi \frac{q}{1-q}$ \rule{0mm}{8mm} \\
$c_2$         & $2\pi I_0(2a)$ & $2r$ & $\displaystyle 4\pi \frac{q^2}{1-3q+2q^2}$ \rule{0mm}{8mm} \\
Fractal Index & $2$ & $1$ & $1-2q$ \rule{0mm}{8mm} \\
\bottomrule
\end{tabular}
\end{center}
\end{table}  

\begin{table}[pt]  
\caption{Values of the parameter $a$ for the von Mises--Fisher kernel,
  the parameter $r$ for the uniform kernel, and the parameter $q$ for
  the power kernel used to generate the planar particles in Figure
  \ref{fig:2d}.  \label{tab:2d}}
\begin{center}
\begin{tabular}{lrrr}
\toprule
Row & $a$ & $r$ & $q$ \\
\midrule
1 &   3 & 1.5 & 0.05 \\
2 &   30 & 1.0 & 0.25 \\
3 &  300 & 0.5 & 0.45 \\
\bottomrule
\end{tabular}
\end{center}
\end{table}

\begin{figure}[pt]  
\begin{center}
\includegraphics[width = \textwidth]{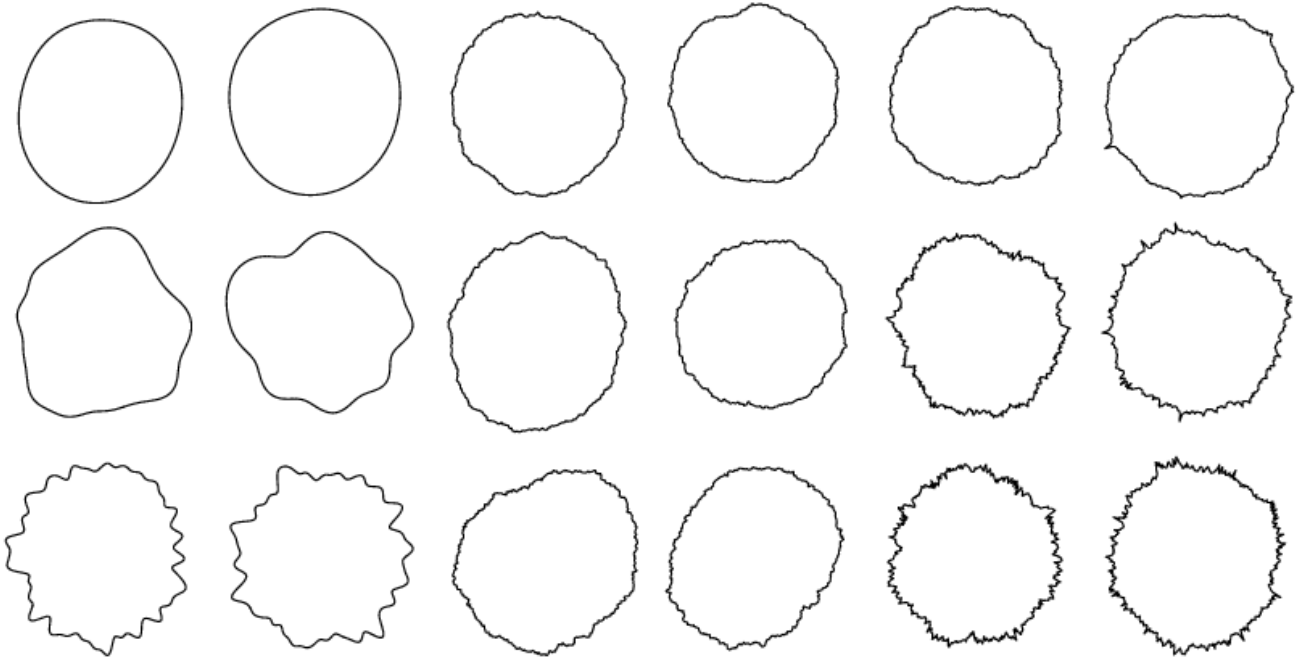}
\caption{Planar particles with mean $\mu_X = 25$ and variance
  $\sigma^2_X = 10$.  Columns 1 and 2 show particles generated using a
  von Mises--Fisher kernel, columns 3 and 4 particles using
  a uniform kernel, and columns 5 and 6 particles using a
  power kernel, with parameters varying by row as described in Table
  \ref{tab:2d}.  The particles in columns 1, 3, and 5 are generated
  under a Gaussian measure, those in columns 2, 4, and 6 under a gamma
  measure.  \label{fig:2d}}
\end{center}
\end{figure}

We now reduce the dimension and consider the planar Gaussian random particle
\[
Y_c = \bigcup_{u \in \Sphere^1} \left\{ o + ru : 0 \leq r \leq \max(X(u),c) \right\}	
\subset \R^2. 
\]
Here $c > 0$, $o \in \R^2$ is an arbitrary centre, and the radial
function $X(u)$ is modelled as
\[
X(u) = \int_{\Sphere^1} K(v,u) \, L(\dd v), \qquad u \in \Sphere^1, 
\]
with a suitable kernel function $K : \Sphere^1 \times \Sphere^1 \to
\bar{\R}$ and a Gaussian measure $L$ on the Borel subsets of the unit
sphere $\Sphere^1 = \{ x \in \R^2 : \| x \| = 1\}$.  We further
construct planar gamma particles where $L$ is a gamma measure on the
Borel subsets of the unit sphere, $L(A) \sim \mathrm{Gamma} \left(
\kappa \, \lambda(A), \tau \right)$ with shape $\kappa > 0$ and rate
$\tau > 0$.  Under this model, $\mu_X = \kappa \, c_1/\tau$ and
$\sigma_X^2 = \kappa \, c_2 / \tau^2$.  The Gamma measure is
independently scattered and we can thus apply the same simulation
method as for the Gaussian particles.

As previously, we assume that the kernel function $K$ is isotropic, in
that $K(v,u) = k(\ang(v,u))$ depends on the points $v, u \in \Sphere
^1$ through their angular or circular distance $\ang(v,u)\in
[0,\pi]$, only.  Table \ref{tab:planar} lists circular analogues of
von Mises--Fisher, uniform, and power kernels along with analytic
expressions for the integrals
\[
c_n = \int_{\Sphere^1} k(\ang(v,u))^n \, \dd v = 2\int_0^\pi k(\eta)^n \, \dd \eta, 
\]
where $n = 1, 2$, and the fractal index, $\alpha$, of the associated
correlation function, as defined in equation \eqref{eq:fractal.index}.
The power kernel model has previously been studied by \citet[Example
  3.3]{Wood1995}.  In analogy to the respective result on $\Sphere^2$,
if $\max_{u \in \Sphere^1} X(u) > c$, the boundary of the Gaussian
particle $Y_c$ has Hausdorff dimension $D = 2 -\frac{\alpha}{2}$
almost surely.

The general form of the associated correlation function is 
\begin{align*}
C(\theta) = \frac{1}{c_2} \Bigg( \int_{\pi-\theta}^\pi k(\phi) &k(2\pi-\phi-\theta) \,\dd\phi
+\int_0^{\pi-\theta} k(\phi) k(\theta+\phi) \,\dd\phi \\
&\quad +\int_0^\theta k(\phi) k(\theta-\phi) \,\dd\phi +\int_\theta^\pi k(\phi) k(\phi-\theta) \,\dd\phi
\Bigg) \! , \qquad 0 \leq \theta \leq \pi.  
\end{align*}
For the von Mises--Fisher kernel with parameter $a > 0$, the
correlation functions admits the closed form
\[
C(\theta) = \frac{I_0 \! \left( a \sqrt{2 (1+\cos \theta)} \right)}{I_0(2a)}, 
\qquad 0 \leq \theta \leq \pi,     
\]
where $I_0$ denotes the modified Bessel function of the first kind and of 
order $0$, and for the uniform kernel with cut-off parameter $r \in
(0,\frac{\pi}{2}]$, we have
\[
C(\theta) = 
\left( 1 - \frac{\theta}{2r} \right) \one(\theta \leq 2r), \qquad 0 \leq \theta \leq \pi.  
\]
For the power kernel with parameter $q \in (0, \tfrac{1}{2})$, tedious
but straightforward computations result in a complex closed form
expression, and a Taylor expansion about the origin yields the fractal
index, $\alpha = 1 - 2q$, stated in Table \ref{tab:planar}.

Thus, the von Mises--Fisher and uniform kernels result in Gaussian
particles with boundaries of Hausdorff dimension 1 and $\tfrac{3}{2}$,
respectively.  Under the power kernel, the Hausdorff dimension of the
Gaussian particle surface is $\tfrac{3}{2} + q$.  Simulated planar
Gaussian and gamma particles with von Mises--Fisher, uniform, and
power kernels are shown in Figure \ref{fig:2d}, with the parameter
values varying by row, as listed in Table \ref{tab:2d}.  The
simulation algorithm of Section \ref{sec:algorithm} continues to apply
with natural adaptions, such as defining the simulation grid $u_m =
2\pi m/M$ for $m = 1, \ldots, M$, where we use $M = 5,\!000$ and $N =
10^5$.

%%%%%%%%%%%%%%%%%%%%%%%%%%%%%%%%%
%      Discussion               %
%%%%%%%%%%%%%%%%%%%%%%%%%%%%%%%%%

\section{Discussion}  \label{sec:discussion}

We have proposed a flexible framework for modelling and simulating
star-shaped Gaussian random particles.  The particles are represented
by their radial function, which is generated by an isotropic kernel
smoothing on the sphere.  From a theoretical perspective, the
construction is perfectly general, as every continuous isotropic
correlation function on a sphere admits an isotropic convolution root
\citep{Ziegel2014}.  The Hausdorff dimension of the particle surface
depends on the behaviour of the associated correlation function at the
origin, as quantified by the fractal index.  Under power kernels we
obtain Gaussian particles with boundaries of any Hausdorff dimension
between 2 and 3.
    
While a non-Gaussian theory remains elusive, we believe that similar
results hold for gamma particles where $L(A) \sim \mathrm{Gamma}
\left( \kappa \, \lambda(A), \tau \right)$ in \eqref{eq:model} with
shape $\kappa > 0$ and rate $\tau > 0$.  For instance, Figures
\ref{fig:Gaussian} and \ref{fig:gamma} show Gaussian and gamma
particles under the power kernel, respectively.  The surface structure
for the different bases resemble each other, even though the particles
exhibit more pronounced spikes under the gamma basis, see also the
planar particles in Figure~\ref{fig:2d}.  Similar particle models may
be generated using different types of \Levy bases $L$, such as Poisson
or inverse Gaussian \citep{Jonsdottir2008}.

\begin{figure}[t]  
\centering
\includegraphics[height=4cm]{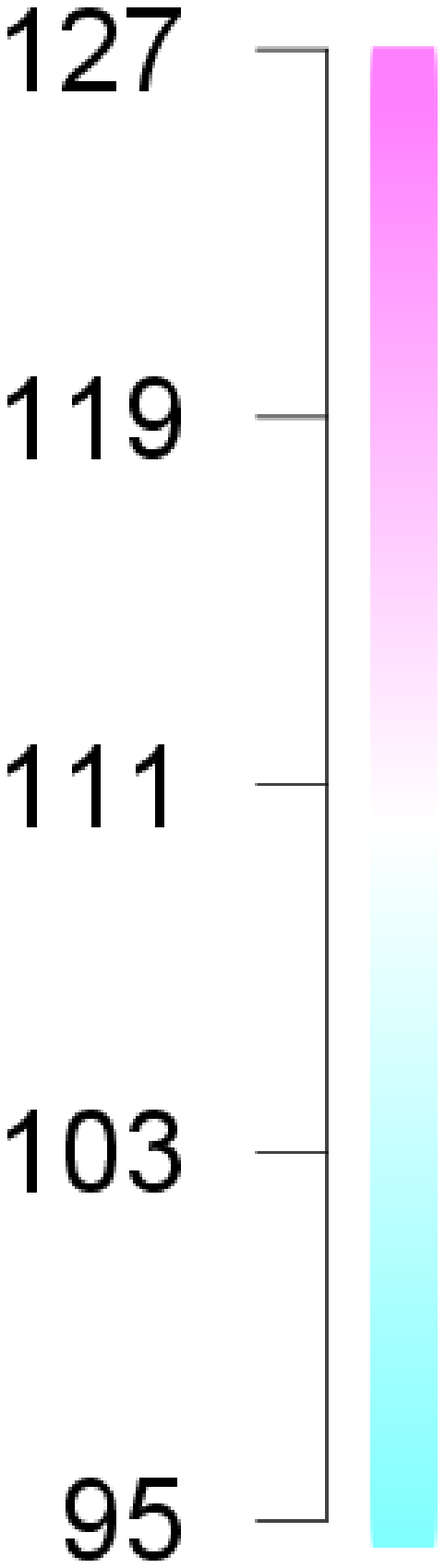}
\hspace{0.3cm}
\includegraphics[height=4cm]{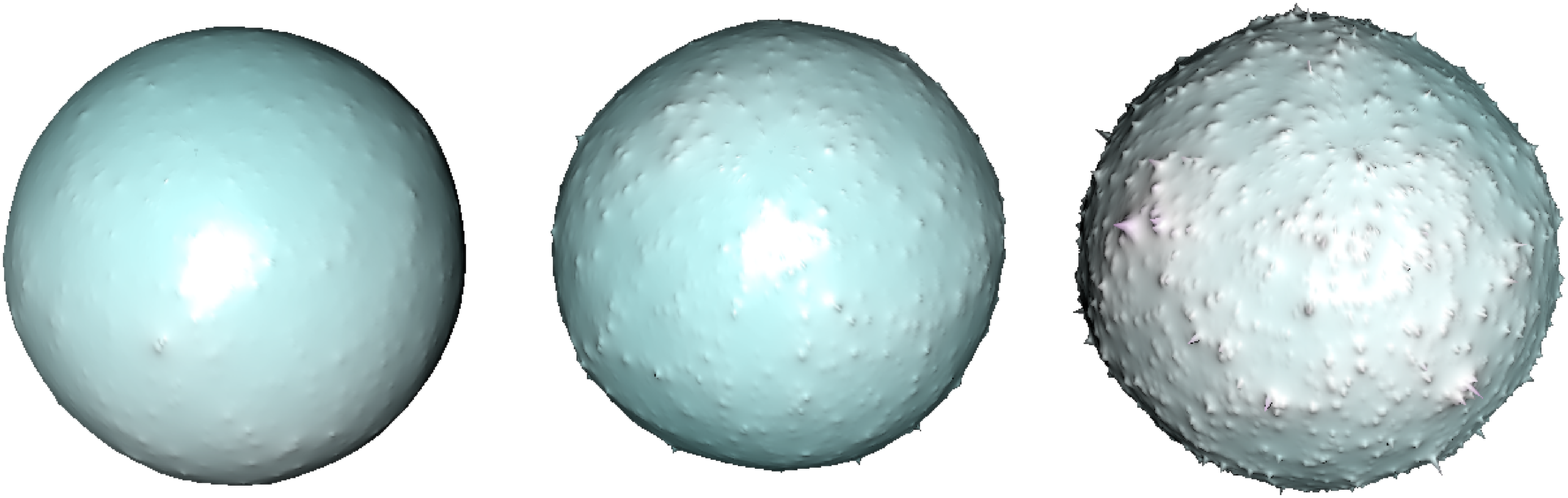}
\caption{Gamma particles with mean $\mu_X = 100$ and variance
  $\sigma^2_X = 10$, using a gamma measure in \eqref{eq:model} and the power kernel
  \eqref{eq:power} with $q = 0.05$ (left), $q = 0.25$ (middle) and $q
  = 0.5$ (right).  \label{fig:gamma}}
\bigskip
\end{figure}

We have focused on three-dimensional particles, except for brief
remarks on planar particles in the preceding section.  However, the
Gaussian particle approach generalises readily, to yield star-shaped
random particles in $\R^d$ for any $d \geq 2$.  The particles are
represented by their radial function and associated with an isotropic
random field on the sphere $\Sphere^{d-1}$.  In this setting,
\citet{EstradeIstas2010} derive a recursion formula that yields closed
form expressions for the isotropic correlation function on
$\Sphere^{d-1}$ that arises under a uniform kernel.  In analogy to
terminology used in the Euclidean case \citep{Gneiting1999b}, we refer
to this correlation function as the `spherical hat' function with
cut-off parameter $r \in (0,\frac{\pi}{2}]$.  Any spherical hat
function has a linear behaviour at the origin, and thus has fractal
index $\alpha = 1$.  \citet{EstradeIstas2010} also show that scale
mixtures of the spherical hat function provide correlation functions
of any desired fractal index $\alpha \in (0,1]$, similarly to the
corresponding results of \citet{HammersleyNelder1955} and
\citet{Gneiting1999b} in the Euclidean case.

A far-reaching, natural extension of our approach uses non-isotropic
kernels to allow for so-called multifractal particles, where the
roughness properties and the Hausdorff dimension may vary locally on
the particle surface.  This fits the framework of
\citet{Gagnonetal2006}, who argue that the topography of Earth is
multifractal, and allows for multifractal simulations of
three-dimensional celestial bodies, as opposed to extant work that
applies to the topography of `Flat Earth'.

We have not discussed parameter estimation under our modelling
approach, leaving this to future work.  In a Bayesian setting,
inference could be performed similarly to the methods developed by
\citet{WolpertIckstadt1998}, who use a construction akin to the random
field model in \eqref{eq:model} to represent the intensity measure of
a spatial point process, and propose a simulated inference framework,
where the model parameters, the underlying random field, and the point
process are updated in turn, conditional on the current state of the
other variables.  Alternatively, \cite{Ziegel2013} proposes a
non-parametric inference framework based on series of Gegenbauer
polynomials.

%%%%%%%%%%%%%%%%%%%%%%%%%%%%%%%%%
%      Acknowledgements         %
%%%%%%%%%%%%%%%%%%%%%%%%%%%%%%%%%

\acks

The authors thank Anders R{\o}nn-Nielsen, Eva B.~Vedel Jensen, Jens
Ledet Jensen, Richard Askey, Werner Ehm, two anonymous reviewers, and
the editor for comments and discussions.  This research has been
supported by the Centre for Stochastic Geometry and Advanced
Bioimaging at Aarhus University, which is funded by a grant from the
Villum Foundation; by the Alfried Krupp von Bohlen und Halbach
Foundation; by the German Research Foundation (DFG) within the
programme ``Spatio-/Temporal Graphical Models and Applications in
Image Analysis'', grant GRK 1653; by Statistics for Innovation,
$sfi^2$, in Oslo; by the U.S. National Science Foundation, grant
DMS-1309808; and by a Romberg Guest Professorship at the Heidelberg
Graduate School for Mathematical and Computational Methods in the
Sciences, funded by the German Universities Excellence Initiative
grant GSC 220/2.

%%%%%%%%%%%%%%%%%%%%%%%%%%%%%%%%%
%      Bibliography             %
%%%%%%%%%%%%%%%%%%%%%%%%%%%%%%%%%

\begin{singlespace}

\bibliography{minbib}

\end{singlespace}

%%%%%%%%%%%%%%%%%%%%%%
%      Appendix      %
%%%%%%%%%%%%%%%%%%%%%%

\appendix

\section{Proof of Theorem \ref{thm:power}} 

We proceed in two parts, demonstrating first the asymptotic expansion
\eqref{eq:as} with the constant $b_q$ in \eqref{eq:bq}, and then
establishing the equality of the expressions in \eqref{eq:bq} and
\eqref{eq:explicit}, which confirms that $b_q$ is strictly positive.
The claim about the fractal index then is immediate from Theorem
\ref{thm:Hausdorff}.

In what follows, if $A(\cdot)$ and $B(\cdot)$ are nonnegative functions
on a common domain, we write
\[
A \lesssim B
\]
if there is a constant $C>0$ such that $A\leq CB$ and $C$ is
independent of any parameters or arguments appearing in $A$ and $B$
when the latter are allowed to vary in their specified domains.

\subsection{Asymptotic expansion \eqref{eq:as}} 

Recall that 
\[
C(\theta) = \frac{2}{c_2} \int_0^\pi \! \left( \pi^q \lambda^{-q} - 1 \right) \sin\lambda 
\int_{A(\lambda)} \! \left( \pi^q a(\theta,\lambda,\phi)^{-q} -1 \right) \dd\phi \, \dd\lambda,
\]
where 
\[
t(\theta,\lambda,\phi) = \sin\theta \sin\lambda \cos\phi + \cos\theta \cos\lambda, 
\qquad
a(\theta,\lambda,\phi) = \arccos t(\theta,\lambda,\phi), 
\]
and 
\[
A(\lambda) = \{ \phi \in [0,\pi] : 0 < a(\theta,\lambda,\phi) \leq \pi \}.  
\]
Therefore,
\[
\frac{c_2}{2} \big( C(0) - C(\theta) \big)
= \int_{0}^\pi \!\! \left( \pi^q \lambda^{-q} - 1 \right) \sin\lambda 
\left\{ \int_0^\pi \left( \pi^q \lambda^{-q} - 1 \right) \! \dd\phi 
- \int_{A(\lambda)} \!\! \left( \pi^{q} a(\theta,\lambda,\phi)^{-q} - 1 \right) 
\! \dd\phi \right\} \! \dd\lambda.
\]
Since $A(\lambda) = [0,\pi]$ for $\lambda \in (0,\pi-\theta]$ and
  $A(\lambda) \subset [0,\pi]$ for $\lambda \in (\pi-\theta,\pi)$, we
  decompose the integral on the right-hand side as $P_{1q}(\theta) +
  P_{2q}(\theta)$, where $P_{1q}(\theta)$ and $P_{2q}(\theta)$
  correspond to the integral with respect to $\lambda$ over
  $(0,\pi-\theta)$ and $(\pi-\theta,\pi)$, respectively. 
The first mean value theorem for integration
  implies that there exists a $t \in (\pi-\theta,\pi)$ such that
\[
P_{2q}(\theta) = \: \theta \left( \pi^q t^{-q} - 1 \right) \sin t \left\{ 
    \int_0^\pi \!\! \left( \pi^q t^{-q} - 1 \right) \! \dd\phi 
    - \int_{A(t)} \!\! \left( \pi^q a(\theta,t,\phi)^{-q} - 1 \right) \! \dd\phi \right\} \! . 
\]
Hence, $P_{2q}(\theta)$ decays at least as fast as
$\mathcal{O}(\theta^2)$ as $\theta \downarrow 0$.

As regards the first term, substituting $\lambda = \theta x$ yields
\[
P_{1q}(\theta)
= \theta^{2-2q} \, \pi^{2q} \int_0^{(\pi-\theta)/\theta} \frac{\sin(\theta x)}{\theta}
\left( x^{-q} - \pi^{-q} \theta^q \right)  
\int_0^\pi \!\! \left( x^{-q} - a(\theta,\theta x,\phi)^{-q} \theta^q \right) \! \dd\phi \, \dd x.
\]
In order to prove the asymptotic behaviour \eqref{eq:as} it now suffices to show 
that 
\begin{equation}  \label{eq:I}
\lim_{\theta \downarrow 0} I(\theta) 
= \frac{c_2}{2 \pi^{2q}} b_q 
= \int_0^{\infty} x^{1-q}f(0,x) \dd x, 
\end{equation}
where 
\[
I(\theta) = \int_0^{(\pi-\theta)/\theta} 
\frac {\sin (\theta x) }\theta (x^{-q}-\pi^{-q}\theta^q) f(\theta,x) \dd x
\]
for $\theta > 0$, with 
\[
f(\theta,x) = \int_0^\pi \left( x^{-q} - a(\theta,\theta x,\phi)^{-q} \theta^q \right) \dd\phi 
\]
for $x > 0$ and $\theta \geq 0$.  As we aim to find the limit
$\lim_{\theta\downarrow 0} I(\theta)$, we may assume that $\theta \in
(0,\theta_0)$ for some $0 < \theta_0 \ll 1$, and that $\lambda \in
[0,\pi-\theta]$.

\begin{lemma}  \label{le:1} 
We have $t(\theta,\lambda,\phi) \leq \cos(\theta-\lambda)$ and
\[
a(\theta, \lambda, \phi)^{-q} \leq \abs{\theta-\lambda}^{-q}.
\]
\end{lemma}

\begin{proof}
We write
\[
t(\theta,\lambda,\phi) 
= \sin\theta \sin\lambda \cos\phi + \cos\theta \cos\lambda 
= \cos (\theta-\lambda) + \sin\theta \sin\lambda (\cos\phi-1).
\]
Since $\cos\phi-1\in[-2,0]$ and the inverse cosine function is
monotonically decreasing, the claims follow.
\end{proof}

\begin{lemma}  \label{le:2}
Define $f(0,x) = \lim_{\theta \downarrow 0} f(\theta,x)$. Then the limit exists and equals 
\[
f(0,x) = \int_0^\pi 
\left( x^{-q} - (x^2 - 2x\cos\phi +1)^{-q/2} \right) \dd \phi
\]
for $x \not \in \{ 0, 1 \}$.
\end{lemma}

\begin{proof}
For $x \not \in \{ 0, 1 \}$ fixed, the integrand in the definition of $f(\theta,x)$ 
is bounded in $\phi$.  The claim follows immediately from the limit
\[
 \lim_{\theta\downarrow 0} 
= \frac{a(\theta,\theta x,\phi)}{\theta} 
= (x^2 - 2x\cos\phi + 1)^{1/2}
\end{equation*}
along with Lebesgue's dominated convergence theorem.  Indeed, noting that 
\[
\frac{\arccos(t)}{\theta} = 
\left. \frac{\arccos \left( 1 - y^2 \right)}{y} \: \frac{y}{\theta} \: \right|_{y = (1-t)^{1/2}}
\]
for $t \in (0,1)$, we find that 
\[
\lim_{\theta \downarrow 0} \frac{a(\theta,\theta x,\phi)}{\theta} 
= \left. \frac{\dd}{\dd y} \arccos \left( 1 - y^2 \right) \right|_{y=0} \, 
\lim_{\theta \downarrow 0} \left( 
  \frac{1 - \cos\theta \cos\theta x}{\theta^2} - \frac{\sin\theta \sin\theta x}{\theta^2} \cos\phi \right)^{1/2}
= \left( x^2 - 2x \cos\phi + 1 \right)^{1/2}. 
\]
\end{proof}

\begin{lemma}  \label{le:3}
We have
\[
\abs{f(\theta,x)} \leq \pi  (x^{-q}+ \abs{x-1}^{-q}),
\]
for $x \in [0,(\pi-\theta)/\theta]$.  
\end{lemma} 

\begin{proof}
We find from Lemma \ref{le:1} that
\[
a(\theta, \theta x, \phi)^{-q}\theta^q \leq \abs{x-1}^{-q},
\]
and the claim follows.
\end{proof}

For later purposes we need to find the Taylor expansion of $a(\theta, \lambda, \phi)^{-q}$ around $\theta=0$,
\[
a(\theta, \lambda, \phi)^{-q} 
= a(0, \lambda, \phi)^{-q} 
+ \frac{\dd}{\dd\theta} a(\theta, \lambda, \phi)^{-q} \Big|_{\theta=0} \theta 
+ R(\theta,\lambda,\phi),
\end{equation*}
where $R$ denotes the error term.

\begin{lemma}  \label{le:4}
We have
\[
a(\theta,y,\phi)^{-q} 
= y^{-q} + \frac{q \cos\phi}{y^{q+1}}\theta + R(\theta,y,\phi),
\]
where the error term satisfies
\[
\abs {R(\theta,y,\phi)} \lesssim \theta^2
\left( \frac{1}{\abs{y-\theta}^{q+2}} + \frac{1}{\abs{y-\theta}^{q+1} \sin(y-\theta)} \right)
\]
for $y \in [2\theta, \pi-\theta]$.
\end{lemma}

\begin{proof}
From $t(0,y,\phi) = \cos y$, the zeroth-order term is immediately seen
to be $y^{-q}$.  For the first-order term, we compute
\[
\frac{\dd}{\dd\theta} a(\theta,y,\phi) 
= - \frac{1}{\sqrt{1-t^2}}(\cos\theta \sin y \cos\phi - \sin\theta \cos y),
\]
so
\[
\frac{\dd}{\dd\theta} a(\theta,y,\phi) \Big|_{\theta=0} = -\cos\phi.
\]
Hence,
\[
\frac{\dd}{\dd\theta} a(\theta,y,\phi)^{-q} \Big|_{\theta=0} 
= - \frac{q}{a(0,y,\phi)^{q+1}} \frac{\dd}{\dd\theta} a(\theta,y,\phi) \Big|_{\theta=0} \\
= \frac{q\cos\phi}{y^{q+1}}.
\]

To estimate $R$, we present it in Lagrange form,
\[
R(\theta,y,\phi) 
= \frac{1}{2} \frac{\dd^2}{\dd\theta^2} a(\theta,y,\phi)^{-q}
  \Big|_{\theta=\hat\theta} \theta^2,
\]
for some $\hat\theta \in [0,\theta]$.  Then, we have
\begin{eqnarray*}
\frac{\dd^2}{\dd\theta^2} a(\theta,y,\phi)^{-q} 
& = & \frac{\dd}{\dd\theta} \left( - \frac{q}{a(\theta,y,\phi)^{q+1}} 
      \frac{\dd}{\dd\theta} a(\theta,y,\phi) \right) \\
& = & \frac{q(q+1)}{a(\theta,y,\phi)^{q+2}} \left( 
      \frac{\dd}{\dd\theta} a(\theta,y,\phi)\right)^2 - 
      \frac{q} {a(\theta,y,\phi)^{q+1}}
      \frac{\dd^2}{\dd\theta^2} a(\theta,y,\phi) \\
& = & A(\theta,y,\phi) + B(\theta,y,\phi). \rule{0mm}{8mm}
\end{eqnarray*}
Let us now estimate $A(\theta,y,\phi)$ from above for $y \in [2\theta,\pi-\theta]$.  From
\[
\frac{1}{\sqrt{1-t^2}} \leq \frac{1}{\sqrt{1-\cos^2(y-\theta)}} = \frac 1 {\sin (y-\theta)}
\]
we get
\[
\abs{\frac{\dd}{\dd\theta} a(\theta,y,\phi)} \leq \frac{\sin y + \sin\theta}{\sin(y-\theta)}.
\]
By monotonicity, we have $\sin\theta \leq \sin(y-\theta)$ and $\sin y \leq 2 \sin(y-\theta)$. 
Hence, in view of Lemma \ref{le:1}, we find that
\[
\abs{A(\theta,y,\phi)} \lesssim \abs{y-\theta}^{-q-2}
\]
for $y \in [2\theta,\pi-\theta]$.

To estimate $B(\theta,y,\phi)$ in the same range, we compute
\[
\frac{\dd^2}{\dd\theta^2} a(\theta,y,\phi) = C(\theta,y,\phi) + D(\theta,y,\phi),
\]
where
\[
C(\theta,y,\phi) 
= \frac{\dd}{\dd\theta} 
    \left( - \frac{1}{\sqrt{1-t^2}} \right) (\cos\theta \sin y \cos\phi - \sin\theta \cos y) 
= - \frac{t}{(1-t^2)^{3/2}} (\cos\theta \sin y \cos\phi - \sin\theta \cos y)^2
\]
and
\[
D(\theta,y,\phi) 
= - \frac{1}{\sqrt{1-t^2}} \frac{\dd}{\dd\theta} (\cos\theta \sin y \cos\phi - \sin\theta \cos y)
= \frac{1}{\sqrt{1-t^2}} (\sin\theta \sin y\cos\phi + \cos \theta \cos y).
\]
We have
\[
\frac{\abs t}{(1-t^2)^{3/2}} 
\leq \frac{\abs{\cos(y-\theta)}}{\sin^3(y-\theta)} 
\leq \frac{1}{\sin^3(y-\theta)}
\]
for $y \in [2\theta,\pi-\theta]$.  Hence,
\[
\abs{C(\theta,y,\phi)} 
\leq \frac{(\sin y + \sin\theta)^2}{\sin^3(y-\theta)} 
\lesssim \frac{1}{\sin(y-\theta)}.
\]
Similarly, $\abs{D(\theta,y,\phi)} \lesssim 1/(\sin(y-\theta))$.  Hence,
\[
\abs{B(\theta,y,\phi)} \lesssim \frac{1}{\abs{y-\theta}^{q+1}\sin(y-\theta)}
\]
for $y \in [2\theta,\pi-\theta]$.

Finally, for $\hat\theta \in [0,\theta]$, 
\[
\abs{A(\hat\theta,y,\phi)} 
\lesssim \frac{1}{(y-\hat \theta)^{q+2}} \leq \frac{1}{\abs{y-\theta}^{q+2}},
\]
and similarly,
\[
\abs{B(\hat\theta,y,\phi)} 
\lesssim \frac{1}{(y-\hat\theta)^{q+1} \sin(y-\hat\theta)} 
\lesssim \frac{1}{\abs{y-\theta}^{q+1} \sin(y-\theta)}.
\]
Combining the estimates for $A$ and $B$, the proof of the lemma is complete.
\end{proof}

In what follows, we need the classical estimate
\begin{equation}  \label{eq:sin}
0 \leq \frac{\sin \theta x}{\theta} \leq x,
\end{equation}
for $x \geq 0$ and $\theta > 0$.

\begin{lemma}  \label{le:5}
We have
\[
\frac{\sin \theta x}{\theta} \abs{f(\theta,x)} \lesssim (x-1)^{-1-q}
\]
for $x \in [2,(\pi-\theta)/\theta]$.
\end{lemma}

\begin{proof}
From Lemma \ref{le:4} and $\int_0^\pi \cos\phi \dd\phi = 0$, we find 
\[
f(\theta,x)= \theta^{q+2}\int_0^\pi R(\hat \theta, \theta x, \phi) \dd\phi.
\]
Using Lemma \ref{le:4} and \eqref{eq:sin}, we get
\[
\frac{\sin \theta x}{\theta} \abs{f(\theta,x)} 
\lesssim \frac{\sin \theta x}{\theta} (x-1)^{-2-q} + \frac{\sin \theta x}{\sin \theta (x-1)} (x-1)^{-1-q}
\lesssim (x-1)^{-1-q}. 
\]
\end{proof}

\begin{lemma}
We have that $x^{1-q}f(0,x)$ is Lebesgue integrable and
\[
\lim_{\theta \downarrow 0} I(\theta) = \int_0^{\infty} x^{1-q}f(0,x) \dd x.
\]
\end{lemma}

\begin{proof}
From \eqref{eq:sin}, Lemma \ref{le:3}, and Lemma \ref{le:5}, we find that
\[
\abs{\frac {\sin (\theta x) }\theta (x^{-q}-\pi^{-q}\theta^q) f(\theta,x)} 
\lesssim
\begin{cases}
x^{1-2q} + \abs{x-1}^{-q},    & \quad x \in [0,2] \\
(x-1)^{-1-2q} + (x-1)^{-1-q}, & \quad x \in [2,\infty)
\end{cases}
\]
uniformly in $\theta \in (0,\theta_0]$, where $0 < \theta_0 \ll1$.
Since the latter function is Lebesgue integrable, the claims follow
from Lebesgue's dominated convergence theorem along with Lemma \ref{le:2}.  
\end{proof}

This completes the proof of \eqref{eq:I} and therefore of the asymptotic relationship 
\eqref{eq:as} with the constant $b_q$ in \eqref{eq:bq}.

\subsection{Equality of the expressions in \eqref{eq:bq} and \eqref{eq:explicit}}

We now show that the constant $b_q$ is strictly positive.  Specifically, we demonstrate the 
equality of the expressions in \eqref{eq:bq} and \eqref{eq:explicit} for $q \in (0,1)$.  
Toward this end, we first prove that
\begin{equation}  \label{eq:bq-representation}
b_q = \frac{\pi^{2q+1} \Gamma(1 - \tfrac{1}{2} q)}{c_2 \Gamma(\tfrac{1}{2} q)} 
\int_0^\infty t^{q-1} \left( 1 - \ee^{-t} \, {}_1F_1(1 - \tfrac1{2}{q}; 1; t) \right) \frac{\dd t}{t},
\end{equation}
where with $(x)_0 = 1$ and $(x)_n = x (x+1)
\cdots (x+n-1)$ for $n = 1, 2, \ldots$, the classical confluent
hypergeometric function \citep[Chapter 13]{DLMF} can be written as
\[
{}_1F_1(a;b;t) = \sum_{k=0}^\infty \frac{(a)_k}{(b)_k} \frac{t^k}{k!}.
\]

We establish this representation as follows.  With a keen eye on the
inner integral in \eqref{eq:bq}, we note that for $x > 0$ and $\phi
\in (0,\pi)$,
\[
x^{-q} = (x^2)^{-q/2} 
= \frac{1}{\Gamma(\tfrac{1}{2} q)} \int_0^\infty \ee^{-tx^2} \, t^{q/2} \, \frac{\dd t}{t},
\]
and 
\[
(1 + x^2 - 2x\cos\phi)^{-q/2} 
= \frac{1}{\Gamma(\tfrac{1}{2} q)} \int_0^\infty \ee^{-t(1+x^2-2x\cos\phi)} \, t^{q/2} \, \frac{\dd t}{t}.
\]
Substituting these formulae into \eqref{eq:bq}, and interchanging the
order of the integration with respect to $\phi$ and $t$, we obtain
\begin{equation}  \label{eq:bq-intermediate1}
b_q = \frac{2\pi^{2q}}{c_2 \Gamma(\tfrac{1}{2} q)} \int_0^\infty x^{1-q} \int_0^\infty t^{q/2} 
\ee^{-tx^2} \int_0^\pi \left( 1 -  \ee^{-t(1-2x\cos\phi)} \right) \dd\phi \, \frac{\dd t}{t} \, \dd x. 
\end{equation}
By well-known formulae, 
\[
\int_0^\pi \ee^{t\cos\phi} \dd\phi = \pi \, I_0(t) = \pi \, {}_0F_1(1; \tfrac{1}{4} t^2), 
\]
where $I_0$ denotes the modified Bessel function of the first kind of
order $0$, and ${}_0F_1$ is a special case of the generalised
hypergeometric series \citep[formulae 10.32.1 and 10.39.9]{DLMF}.
Therefore,
\[
\int_0^\pi \left( 1 - \ee^{-t(1-2x\cos\phi)} \right) \dd\phi 
= \pi \left( 1 - \ee^{-t} \, {}_0F_1(1;t^2x^2) \right) \! .
\]
Substituting this result into \eqref{eq:bq-intermediate1}, and
interchanging the order of the integration, we obtain
\begin{equation}  \label{eq:bq-intermediate2}
b_q = \frac{2 \pi^{2q+1}}{c_2 \Gamma(\tfrac{1}{2} q)} 
      \int_0^\infty t^{q/2} \int_0^\infty x^{1-q} \ee^{-tx^2} 
      \left( 1 - \ee^{-t} \, {}_0F_1(1;t^2x^2) \right) \dd x \: \frac{\dd t}{t}. 
\end{equation}
With the substitution $x = u^{1/2}$, we get 
\[
\int_0^\infty x^{1-q} \ee^{-tx^2} \dd x 
= \frac{1}{2} \, t^{\tfrac1{2}{q} - 1} \Gamma(1 - \tfrac{1}{2} q).
\]
We apply next a well-known formula for the Laplace transforms of
generalised hypergeometric series \citep[formula 16.5.3]{DLMF} to
obtain, with the substitution $x = t^{-1/2} \, u^{1/2}$,
\[
\int_0^\infty \!\!\! x^{1-q} \, \ee^{-tx^2} {}_0F_1(1; t^2x^2) \dd x 
= \frac{1}{2} \, t^{\tfrac{1}{2} q - 1} \Gamma(1 - \tfrac{1}{2} q) \, {}_1F_1(1 - \tfrac{1}{2} q; 1; t).
\]
Consequently, by \eqref{eq:bq-intermediate2}, we have established the
representation \eqref{eq:bq-representation}.

Finally, we apply the Kummer formula for the ${}_1F_1$ function
\citep[formula 13.2.39]{DLMF} to show that
\[
1 - \ee^{-t} {}_1F_1(1 - \tfrac{1}{2} q; 1; t) 
= \frac{1}{2} qt \, {}_2F_2(\tfrac{1}{2} q + 1, 1; 2, 2; -t).
\]
Thus, by the representation \eqref{eq:bq-representation},
\[
b_q = \frac{\pi^{2q+1} q}{2 c_2} \frac{\Gamma(1 - \tfrac{1}{2} q)}{\Gamma(\tfrac{1}{2} q)} 
\int_0^\infty \!\! t^q \, {}_2F_2(\tfrac{1}{2} q+1, 1; 2, 2; -t) \, \frac{\dd t}{t},
\]
and this integral is a well-known Mellin transform; see \citet[formula
  16.5.1]{DLMF}, where the integral is given in inverse Mellin
transform format \citep[Section 1.14(iv)]{DLMF}.  The proof of the
equality of the expressions in \eqref{eq:bq} and \eqref{eq:explicit}
is now complete. 

%%%%%%%%%%%%%%%%%%%%%%%%%%%%%%%%%
%      The End !!               %
%%%%%%%%%%%%%%%%%%%%%%%%%%%%%%%%%

\end{document}